\documentclass[11pt,reqno]{amsart}
\usepackage[T1]{fontenc}
\usepackage{amsthm}
\usepackage{amsmath,amsfonts,amssymb,mathrsfs}
\usepackage{lmodern}
\usepackage[all]{xy}

\newtheorem{thm}{Theorem}[section]
\newtheorem{lem}[thm]{Lemma}
\newtheorem{proposition}[thm]{Proposition}
\newtheorem{defin}{Definition}
\newtheorem{corol}[thm]{Corollary}
\newtheorem{ex}{ex}[section]

\newenvironment{lemma}{\begin{lem}\;}{\end{lem}}

\newenvironment{theorem}{\begin{thm}\;}{\end{thm}}

\theoremstyle{remark}
\newtheorem{remark}[thm]{\bf Remark}
\newtheorem{example}[ex]{\bf Example}

\mathchardef\mhyphen="2D

\newcommand{\inv}{{}^{-1}}
\newcommand{\isomor}{\,\raisebox{4pt}{$\sim$}{\kern -.89em\to}\,}

\newcommand{\cc}{\mathfrak c}

\newcommand{\frg}{\mathfrak g}

\newcommand{\fra}{\mathfrak a}

\newcommand{\CC}{\mathbb C}
\newcommand{\RR}{\mathbb R}
\newcommand{\ZZ}{\mathbb Z}

\newcommand{\QQ}{\mathbb Q}

\newcommand{\regg}{{\rm Reg}(G)}

\begin{document}

\title{On the density of images of the power maps in Lie groups}
\author[S. Bhaumik]{Saurav Bhaumik}

% \address{Department of Mathematics, Indian Institute of Technology Bombay,  
%   Powai, Mumbai 400076, India}  
%  \email{saurav@math.iitb.ac.in} 

\author[A. Mandal]{Arunava Mandal}
 
% \address{Department of Mathematics, Indian Institute of Technology Bombay, 
%   Powai, Mumbai 400076, India}  
%  \email{amandal@math.iitb.ac.in} 


\date{}

\begin{abstract}Let $G$ be a connected Lie group. 
In this paper, we study the density of the images of individual power maps $P_k:G\to G:g\mapsto g^k$. 
We give criteria for the density 
of $P_k(G)$ in terms of regular elements, as well as Cartan subgroups. 
In fact, we prove that if $\regg$ is the set of regular elements of $G$, then $P_k(G)\cap\regg$ is closed in $\regg$.
On the other hand, the weak exponentiality of $G$ turns out to be equivalent to the density of all the power maps $P_k$.
In linear Lie groups, weak exponentiality reduces to the density of $P_2(G)$. We also prove that
the density of the image of $P_k$ for $G$
implies the same for 
any connected full rank subgroup. 
\end{abstract}

\maketitle
  \section{Introduction}

Let $G$ be a connected Lie group and let $\mathfrak{g}$ be the Lie algebra associated to $G.$
The question of whether the exponential map $\exp:\frg\to G$ is surjective or 
has dense image has been addressed by many authors (see \cite{H}, \cite{H-M}, etc). 
A connected Lie group $G$ is called \emph{weakly exponential} (resp. \emph{exponential}) 
if the image of the exponential map $\exp$ is dense (resp. surjective).
McCrudden \cite{Mc} showed that $\exp:\frg\to G$ is surjective if and only if the \emph{power map} $P_k:G\to G$, defined by
$g\to g^k$ is surjective for all $k\ge 2.$
 The surjectivity of the power maps has been studied by P. Chatterjee (\cite{C1},  
\cite{C2}, \cite{C3}, and see there references), R. Steinberg \cite{St}, Dani and Mandal \cite{D-M}. If a connected Lie group $G$ is weakly exponential, then $P_k(G)$ is dense for all $k\ge 1$. 

For an algebraic group over $\mathbb{R}$ or $\mathbb{C}$, the power maps can have dense images without being surjective.
For a connected complex algebraic group which is not exponential (e.g., $\rm SL(2,\mathbb{C})$), 
there is some $k>1$ for which $P_k$
is not surjective, but have dense images. One can take Weil restriction to produce such an example over $\mathbb{R}.$
For an algebraic group $G$ over $\mathbb{Q}_p$, the question of dense images of the power map is 
equivalent to the study of surjectivity (see Remark \ref{algebraic case}). The surjectivity of $P_k$ in 
$G(\mathbb{Q}_p)$ ($\mathbb{Q}_p$-points of the algebraic group $G$ defined over $\mathbb{Q}_p$)
has been settled by P. Chatterjee in \cite{C2}.
 However there does not seem to be any study in the literature on the density of the individual power maps.

This motivated us to consider the following natural question. Given $k>1$, when is the image $P_k(G)$ dense in $G$?

On the other hand, weak exponentiality of a Lie group is closely related to the regular elements and Cartan subgroups.
A. Borel showed that
a connected semisimple Lie group is weakly exponential if and only if all of its Cartan subgroups are connected
(see \cite[Theorem 2.10]{H-M}).
Later, K.H. Hofmann showed that for a connected Lie group $G$ (not necessarily semisimple), the set of its
regular elements $\regg$ has
a property that, $\regg\cap \exp(\frg)$ is closed in $\regg$ (see \cite[Theorem 17]{H}). Also, he deduced criteria
for weak exponentiality of $G$ in terms of regular elements and Cartan subgroups (see \cite[Corollary 18]{H}).
Following this, K. H. Neeb proved that, a connected Lie group is weakly exponential if and only if all of its
Cartan subgroups are connected (see \cite[Theorem I.2]{N}).

It is therefore natural to ask whether $P_k(G)\cap \regg$ is closed in $\regg$ for all $k\ge 1$.

In the following we answer these questions.

\begin{theorem}\label{dense criteria}Let $G$ be a connected Lie group. Let $k>1$ be an integer. 
The following are equivalent:
\begin{enumerate}
 \item[\rm (a)]$P_k(G)$ is dense in $G$.
\item[\rm (b)] $\regg\subset P_k(G)$.
\item[\rm (c)] If $C$ is a Cartan subgroup of $G$, then $P_k(C)=C$.
\end{enumerate}
\end{theorem}

\begin{theorem}\label{pk closed} 
Let $G$ be a connected Lie group. Let $k>1$ be an integer.
Then $P_k(G)\cap \regg$ is closed in $\regg$.
\end{theorem}

The proof of Theorem \ref{dense criteria} relies crucially on Theorem \ref{pk closed}. A result of W\"ustner, that every Cartan subgroup is compatible with some Levi decomposition, also turned out to be very useful in many of our proofs.

Analogous results due to Hofmann (\cite[Theorem 17, Corollary 18]{H}), and K. H. Neeb (\cite[Theorem I.2]{N}) 
in the context of exponential maps, can be
deduced from Theorem \ref{pk closed} and Theorem \ref{dense criteria} (see Corollary \ref{Cartan subgroup}). 

In analogy to McCrudden's criterion for exponential images for a connected Lie group, 
in terms of divisibility, we have the following corollary.
\begin{corol}\label{density and w.e}
 Let $G$ be a connected Lie group. Then $P_k:G\rightarrow G$ is dense for all $k$ if and only if $G$ is weakly exponential.
 \end{corol}

 The next result also gives a characterization of the density of $P_k.$
\begin{corol}\label{dense criteria 1}
 Let $G$ be a connected Lie group and let $C$ be a Cartan subgroup of $G.$ Let $C^*$ denote the connected component of $C$
 and $k>1$ be an integer.
 Then $P_k:G\rightarrow G$ is dense if and only if
 $P_k:C/C^*\rightarrow C/C^*$ is surjective. In particular, $G$ is weakly exponential if and only if $C/C^*$
 is divisible for any Cartan subgroup $C.$
\end{corol}

The density of $k$-th power map $P_k$ only depends on
its Levi part of the group $G.$ More precisely, if $R$ is the solvable radical (i.e., maximal solvable connected normal
subgroup) of $G$, then $P_k(G)$ is dense in $G$ if and only if $P_k(G/R)$ is dense in $G/R$ 
(Proposition \ref{density on levi part}).
% Thus it reduces to study the density of $P_k$ for semisimple Lie groups only.
We list some conditions for the density of
the images of power maps $P_k$ for simple Lie groups in \S 4. 

The following provides a criterion for weak exponentiality in terms of $P_2$ only.

\begin{corol}\label{square map}
 Let $G$ be a connected linear Lie group. Then $G$ is weakly exponential if and only if $P_2:G\rightarrow G$ is dense.
\end{corol}

This means that (compare with Theorem \ref{dense criteria}), 
in a connected linear Lie group $G$,
 if every element that belongs to a Cartan subgroup admits a square root then $G$ is weakly 
 exponential. For linear groups, this extends Neeb's result, which says that $G$ is weakly exponential if and only if all
 of its Cartan subgroups are connected. On the other hand it strengthens one of P. Chatterjee's results (\cite[Theorem 1.6]{C3}). 
 However, the corollary is no longer valid for non-linearizable groups, as we observe in Example \ref{E}.

The following can be
thought of as a power map analogue of Neeb's result (see \cite[Proposition I.6]{N}).

\begin{theorem}\label{full rank}
 Let $G$ be a connected Lie group and let $A$ be a connected full rank subgroup of $G.$ Let $k>1$ be an integer.
 Then $P_k:G\rightarrow G$ is dense implies
 $P_k:A\rightarrow A$ is dense.
 \end{theorem}
 Let $G$ be a connected linear reductive Lie group and $\mathfrak g$ be its associated Lie algebra. Let $\theta$ be the 
Cartan involution on $\mathfrak g.$ Then there is a natural Cartan decomposition $\mathfrak g=\mathfrak k+\mathfrak p$,
where $\mathfrak k$ and $\mathfrak p$ are the eigenspaces corresponding to the eigenvalues $1$ and $-1$ respectively.
It follows that $[\mathfrak k,\mathfrak k]\subseteq\mathfrak k$, $[\mathfrak k,\mathfrak p]\subseteq\mathfrak p$,
$[\mathfrak p,\mathfrak p]\subseteq\mathfrak k$ and $\mathfrak z(\mathfrak g)\subseteq\mathfrak k.$
Let $\mathfrak a$ be a maximal abelian subspace of $\mathfrak p$ and take a Cartan subalgebra $t_a$ of
$\mathfrak z_{\mathfrak k}(\mathfrak a).$ Define $\mathfrak c:=t_a+\mathfrak a.$
It is noted in Neeb \cite{N} that 
  $\mathfrak c$ is a Cartan subalgebra of $\mathfrak g$ and every Cartan
subalgebra $\mathfrak {c'}$ of $\mathfrak g$ is conjugate to a Cartan subalgebra is of the form
$t'_a +\mathfrak {a'}$, where $\mathfrak a' \subseteq \mathfrak a$ and $t'_a\supseteq t_a.$

We give below another criterion for the density of images of power maps in case of reductive groups,
which is mostly important for non-split case. 

\begin{corol}\label{dense criteria 2}
  Let $G$ be a connected reductive Lie group. Let $k>1$ be an integer.
  Then $P_k:G\rightarrow G$ is dense if and only if $P_k:Z_G(t_a)\rightarrow Z_G(t_a)$
  is dense.
 \end{corol}

The paper is organized as follows. We recall some definitions and prove some results on regular elements, Cartan subgroups, including Theorem \ref{pk closed}, in \S 2. Theorem \ref{dense criteria}, Corollaries \ref{density and w.e}, \ref{dense criteria 1} are proved in \S3.
We discuss the density of power maps in simple Lie groups in \S 4. Corollary \ref{square map} is proved in \S 5.
Finally, in \S 6 we prove Theorem \ref{full rank} and Corollary \ref{dense criteria 2}.

\section{Regular elements, Cartan subgroups and power maps}

In this section we recall some definitions, introduce notations, prove a few lemmas and
Theorem \ref{pk closed}. 

\begin{defin}\label{regular hoffman} \rm\cite{H} An element $g$ in a  
Lie group $G$ is called \emph{regular} if the nilspace 
$N({\rm Ad}_g-1)\subset \frg$ is of minimal possible dimension. The set of regular elements of $G$ is denoted by $\rm Reg(G).$
\end{defin}
\begin{defin}\rm\cite{C1}
Let $G$ be a Lie group. An element $g$ in $G$ is called $P_k$-\emph{regular} if $dP_k$ is nonsingular at $g.$
 \end{defin}

\begin{lemma}\label{P_k regular}
Let $G$ be a real Lie group. Suppose $g,h\in G$, and $h^k=g$. Then $g$ is regular if and only if 
$h$ is both regular and $P_k$-regular.
\end{lemma}
\begin{proof}
Let $g$ be regular. Writing $T={\rm Ad}_g-1$, $U={\rm Ad}_h-1$ and 
\[V={\rm Ad}_{h^{(k-1)}}+{\rm Ad}_{h^{(k-2)}}+\ldots+{\rm Ad}_h+1,\]
we get $T=UV$. Then $N(U)\subset N(T)$, 
and by minimality of $\dim(N(T))$, we conclude that $h$ is regular.

Now, the complexification $\frg_{\mathbb C}$ of the real Lie algebra $\frg$, splits into direct sum of 
generalized eigenspaces $V_\alpha$ under ${\rm Ad}_h$, i.e., 
$\frg_{\mathbb C}=\oplus_{\alpha\in\Delta} V_\alpha$ (where $\Delta$ denotes the set of eigenvalues of ${\rm Ad}_h$).
We note that $V_1=N(U)\otimes \CC$. Let $\alpha^k=1$ but $\alpha\ne 1$. 
Then in the splitting of
$\frg_c$ under ${\rm Ad}_h^k={\rm Ad}_g$, the generalized eigenspace for $1$ will contain $V_1\oplus V_\alpha$. By minimality of 
$\dim_\CC(V_1)=\dim_\RR(N(U))$, $V_\alpha=0$. By \cite[Lemma 2.1]{C1}, it follows that $h$ is $P_k$-regular.

Conversely, let $h$ be $P_k$-regular and regular. By \cite[Lemma 2.1]{C1}, $V$ is invertible
on $\frg_{\mathbb C}$ and $N(U)\otimes \CC=N(T)\otimes \CC$ and hence $N(T)$ is of minimal dimension.
\end{proof}

We now recall the definition of Cartan subgroup of a Lie group.
Let $G$ be a connected Lie group with its Lie algebra $\frg.$ Let $\mathfrak c$ be the Cartan subalgebra of 
$\frg$ and let $\Delta$ be the set of roots of $\frg_{\mathbb C}$
belonging to $\mathfrak c_{\mathbb C}.$ 
Thus $\frg_{\mathbb C}=\mathfrak c_{\mathbb C}+\Sigma_{\alpha\in\Delta}\frg_{\mathbb C}^{\alpha}.$
The normalizer $N_G(\mathfrak c)$ 
of $\mathfrak c$ is defined by
 $N_G(\mathfrak c):=\{g\in G|{\rm Ad}_g(\mathfrak c)=\mathfrak c\}.$ Note that $N_G(\mathfrak c)$ is a closed subgroup of $G$ with Lie algebra 
 $\mathfrak n_{\frg}(\mathfrak c)=\{X\in\frg|[X,\mathfrak c]\subseteq\mathfrak c\}.$ Let
 $C(\mathfrak c):=\{g\in N_G(\mathfrak c)|\alpha\circ{\rm Ad}_g|_{\mathfrak c_{\mathbb C}}=\alpha, \forall\alpha\in\Delta\}.$ 
 %Then it turns out that $C(\mathfrak c)=\{g\in N_G(\mathfrak c)|{\rm Ad}_g\circ\frg_{\mathbb C}^{\alpha}=\frg_{\mathbb C}^\alpha\}.$

 A closed subgroup $C$ of a Lie group $G$ is called a \emph{Cartan subgroup} if its Lie algebra $\mathfrak c$ is a Cartan subalgebra 
 of $\frg$ and $C(\mathfrak c)=C.$ This agrees with the usual definition of a Cartan subgroup (see Appendix of \cite{N}). In particular, a Cartan subgroup is nilpotent.
 
 Let us recall a theorem due to M. W\"ustner (\cite[Theorem 1.9 (ii)]{W}). 
  Let $G$ be a connected Lie group and let $R$ be its solvable radical.
 Then for any Cartan subgroup $C$ of $G$, there is a Levi subgroup $S$ of $G$ such that $C$ 
 can be decomposed as $C=(C\cap S)(C\cap R).$ Moreover, $C\cap S$ is a Cartan subgroup of $S$,
 $C\cap R$ is connected and $C\cap R\subseteq Z_G(S).$
 
 Now we can immediately observe that, if $G$ is connected linear Lie group then any Cartan subgroup can be written as
 $C=TN$ (almost direct product), where $T$ is abelian and $N$ is connected nilpotent subgroup of $C.$ Indeed, if $G$ is linear, so is $S$. This means, $S$ is isomorphic to the identity component (in Euclidean topology) of a geometrically connected semisimple algebraic group defined over $\RR$. In this case, $T=C\cap S$ is abelian, and $N=C\cap R$ is connected nilpotent. This will be used later.
 
 \begin{lemma}\label{A} 
 Let $G$ be a Lie group, and let $Z\subset Z(G)$ be a closed central subgroup. Let $G_1=G/Z$, 
 and let $\pi:G\to G_1$ be the projection. 

(a) If $C$ is a Cartan subgroup of $G$, then $C_1=\pi(C)$ is a Cartan subgroup of $G_1$, and $C=\pi\inv(C_1)$.

(b) Suppose $Z$ is connected. If $P_k(C_1)$ is both closed and open in $C_1$, then $P_k(C)$ is both closed and open in $C$.

\end{lemma}

\proof $(a)$ It follows from the description of Cartan subgroup in terms of its adjoint action on the Lie algebra. 
Indeed, if $\mathfrak c$ is a Cartan subalgebra of $\frg$, then $\mathfrak c_1=\mathfrak c/\mathfrak z$ is a 
Cartan subalgebra of $\frg_1=\frg/\mathfrak z$,  
while $\mathfrak z\subset \mathfrak z(\frg)\subset \mathfrak z_\frg(\mathfrak c)$. For any $x\in G$, ${\rm Ad}_x$ preserves
$\mathfrak z(\frg)$, hence $\mathfrak z$. Therefore ${\rm Ad}_{\pi(x)}$ normalizes $\mathfrak c_1$ 
if and only if ${\rm Ad}_x$ normalizes 
$\mathfrak c$. The roots $\overline{\alpha}$ of $\frg_1$ with respect to $\mathfrak c_1$ are induced from roots $\alpha$ of $\frg$ 
with respect to $\mathfrak c$, and $\frg^{\alpha}=\frg_1^{\overline{\alpha}}$. Hence $C_1$ is a Cartan subgroup of $G_1.$

Again, $C\subset \pi\inv(C_1)$. Let $x\in \pi\inv(C_1)$. Since $C_1=\pi(C)$, there is some $y\in C$ such that 
$\pi(x)=\pi(y)$. This means, $x=yz$, for some $z\in Z\subset Z(G)\subset C$, which proves $C=\pi^{-1}(C_1).$

$(b)$ By (a), we get a short exact sequence $1\to Z\to C\stackrel \pi\to C_1\to 1$. 
Now, since $Z$ is a connected abelian group, for any $x\in C$, $xZ\cap P_k(C)$ is nonempty if and only if
$xZ\subset P_k(C)$. Therefore $P_k(C)=\pi\inv (P_k(C_1))$.\hfill$\Box$
\begin{lemma}\label{B}
 Let $C$ be a Lie group and let $\tilde C$ be a cover of $C.$ If $P_k(\tilde C)$ is both open and closed
 then so is $P_k(C).$
\end{lemma}
\begin{proof}
Let $\pi:\tilde C\rightarrow C$ be the covering map. Then $\pi(P_k(\tilde C))=P_k(C).$ By hypothesis, $P_k(\tilde C)$ is disjoint
union of connected components of $\tilde C.$ Let $P_k(\tilde C)=\sqcup_{i\in I}\tilde X_i$, where $\tilde X_i$'s are
connected component of $\tilde C$ and $I$ is the indexing set.
For any $i$, $\pi(\tilde X_i)$ is connected, hence there is a connected component $X_i$ of $C$ such that $\pi(\tilde X_i)\subseteq X_i$. 
We claim that for $i\in I$, $\pi(\tilde X_i)$ is a connected component of $C.$
 Suppose the claim holds, then $\pi(\tilde X_i)=X_i$ and
$\pi(P_k(\tilde C))=\sqcup_{i\in I}\pi(\tilde X)=\sqcup_{i\in I} X_i.$ 
As $X_i$'s are both open and closed in $C$,
$P_k(C)$ is both open and closed. This will prove the lemma.

To prove our claim, we note that, as
$\pi$ is a covering map, $\pi(\tilde X_i)$ is open in $X_i$, therefore it is enough to 
prove that $\pi(\tilde X_i)$ is closed in $X_i.$
Now, let $x$ be an element in the closure of $\pi(\tilde X_i)$. As $\pi$ is covering of a Lie group $C$, there is a \emph{connected}
open set $U$ containing $x$ such that $\pi^{-1}(U)=\sqcup_r V_r$, where $V_r$'s are open in $\tilde X$ and
$\pi$ restricted to $V_r$ is a diffeomorphism onto $U.$ Then, there exists $r_0$, such that $\tilde X_i\cap V_{r_0}\neq \emptyset$. 
Since $V_{r_0}$ is connected, it must lie entirely in the connected component $\tilde X_i$. Therefore $x\in U=\pi( V_{r_0})\subset \pi(\tilde X_i)$.
\end{proof}

 \begin{proposition}\label{cartan pk closed} 
 Let $G$ be any connected real Lie group and let $C$ be a Cartan subgroup of $G$. For any $k\ge 0$, $P_k(C)$ 
 is both open and closed in $C$.
\end{proposition}
\begin{proof} 
First note that it is enough to consider the case when $Z(G)$ is discrete. Indeed, let $G_1=G/Z(G)^*$, 
where $Z(G)^*$ is the identity component of $Z(G)$, and let $\pi_1:G\to G_1$ be the projection. 
Then by Lemma \ref{A}, it is enough to prove that $P_k(C_1)$ is closed and open 
in $C_1$. Inductively,  $G_{i+1}=G_i/Z(G_i)^*$, and let $C_{i+1}$ be the image of $C_i$ in $G_{i+1}$. 
Since the dimension of $G$ is finite, there is 
some $i$ such that $Z(G_i)$ is discrete. Henceforth, we will assume that $Z=Z(G)$ is discrete.

Note that $G$ is a covering group of the linear group $G'=G/Z$. Let $\pi:G\to G'$ denote the covering projection. 
 Therefore, by Lemma \ref{A}, $\pi(C)$ is a Cartan 
subgroup of $G'$. Then by W\"ustner, there is a Levi subgroup $S$ of $G'$ such that 
$\pi(C)=(\pi(C)\cap S)(\pi(C)\cap R)$ (almost direct product), where $R$ is the solvable radical of $G'.$

Now, write $G'=S\cdot R$ (almost semidirect product). As $S$ acts on $R$, we consider an external semidirect product
of $S$ and $R$, denoted by $S\ltimes R.$
We observe that there is a
surjective group homomorphism $P:S\ltimes R\rightarrow S\cdot R$, whose kernel is isomorphic to
$D=\{(z,z^{-1})|z\in S\cap R\}.$ Since $D$ is discrete normal, it is central in $S\ltimes R.$
Hence by Lemma \ref{A}, $P^{-1}(\pi(C))$ is a Cartan subgroup of $S\ltimes R$ and in fact
$P^{-1}(\pi(C))=(\pi(C)\cap S)\times(\pi(C)\cap R).$ 
As $\pi(C)\cap R$ is connected linear
nilpotent, it can be written as $T_1\times N$ for some torus $T_1$ and
simply connected subgroup $N$ of $\pi(C)\cap R.$ Therefore
$P^{-1}(\pi(C))$ is of the form $T\times N$, where $T=(\pi(C)\cap S)\times T_1.$ Hence
 $P_k(P^{-1}(\pi(C))$ is both open and closed in $P^{-1}(\pi(C))$ implies $P_k(\pi(C))$ is so
 in $\pi(C)$, as $S\ltimes R$ is a finite covering of $G'.$
 Therefore, to prove $P_k(\pi(C))$ is both open and closed, it is enough to show the same for $P_k(P^{-1}(\pi(C))).$

Hence, by Lemma \ref{B}, we may assume that
$\pi(C)=T\times N$, where $T$ is an abelian group and $N$ is a simply connected nilpotent group.
Now, $\pi:C\to T\times N$ is a covering projection. Therefore $C$ has to be of the form $A\times N$,
where $A$ is a covering space of $T$.
Indeed, let $H_1=\pi\inv(T\times \{1\})$ and $H_2=\pi\inv(\{1\}\times N)$, where $\{1\}$ denotes the identity element in $G.$
Then $H_1$, $H_2$ are normal in $C$. Again, $N$ is simply connected, 
so the covering space $H_2\to N$ has a section, while the kernel $Z$ is central. 
Hence $H_2=N\times Z$. Now, let $x\in H_1$. Then $\iota_x$
induces an action on $H_2$ which preserves the central $Z$. Let $(y,z)\in N\times Z$. Then $\iota_x(y,z)=(y',z')$. 
Now, taking projection to $N$, since $x\in H_1$, we see that $y=y'$. Hence the action of $\iota_x$ is given by a matrix 

\[\left(\begin{array}{ll} id_N & \phi\\ \{1\} & id_Z\end{array}\right)\]

\noindent where $\phi$ is a homomorphism $N\to Z$. Now, since $N$ is connected and $Z$ is discrete, $\phi$ is zero, 
hence $H_1$ and $H_2$ commute. Therefore we have a surjective homomorphism of groups $H_1\times H_2\to C$ given by 
inclusions and multiplication in $C$. The kernel is $\{(z,z^{-1})|z\in Z\}$. Since $H_2=N\times Z$, we see that the induced 
homomorphism $H_1\times N\to C$ is an isomorphism. We set, $A=H_1$.

This means, it is enough to prove that $P_k(A)$ is closed in $A$. So we can assume that $C$ is a covering group of $T$ with
central kernel $Z$.

Since $G'$ is linear, $T$ is of the form $F\times T_2$, where $F$ is a finite $\ZZ/2\ZZ$-vector space, 
and $T_2$ is a connected abelian group. Now, in order to prove that $P_k(C)$ is both open and closed in $C$, by Lemma \ref{B},
it is enough to do so after replacing $C$ by a covering group. In particular, we can replace $T_2$ by its simply connected 
covering, which will be a vector space, say $V$, and replace $C$ by the fiber product of $V$ and $C$ over $T_2.$

So we can assume that $C$ is an extension of $F\times V$ by the central $Z$.

Let $p:C\to F$ be the composite $C\stackrel\pi\to {F\times V}\stackrel {Pr_1}\to F.$
Let us consider the normal subgroup $B=\pi\inv(0\times V)$. This is a covering group of the simply connected $V$ by $Z$,
hence is of the form $Z\times V$. Since $V$ is connected abelian and $Z$ is discrete, $B$ is central by similar argument
as above.

Therefore, we can express $C$ as an extension of $F$ by the central $B=Z\times V$ as follows.
\[0\to B\to C\stackrel p\to F\to 0.\]

Let the number of elements in $F$ be $r$, and let $F=\{p(x_1),\ldots,p(x_r)\}$. Now, $C$ is the disjoint union of the 
cosets $x_iB$. Therefore, $P_k(C)$ is the union of $P_k(x_iB)$. Note that, since $B$ is central in $C$, $P_k(xB)=x^kP_k(B)$, and
hence is closed and open in $x^kB$. This implies that $P_k(C)=\cup_{i=1}^r P_k(x_iB)$ is closed and open in $C$. 
\end{proof}

\noindent{\it Proof of Theorem~\ref{pk closed}}: By \cite[Ch. VII, \S 4 n${}^\circ$ 2, Proposition 5]{B},
for each regular $g\in G$ and for each 
identity neighbourhood $W$ there is a neighbourhood $V$ of $g$ such that for each $x\in V$ there is a $w\in W$ with
$C(x)=\iota_w C(g)$ (equivalently, $\cc(x)={\rm Ad}_w\cc(g)$ at the level of Lie subalgebras). 

Let $g\in \regg$ and let $g_n\in\regg$ be a sequence such that $g_n=h_n^k$ and $g$ is the limit of the sequence
i.e. $g_n\to g$ in $G$. Note that $P_k:G\to G$ is nonsingular at identity, because $dP_k$ at identity is 
multiplication by $k$. Hence there are sequences of identity neighbourhoods $W_i$, $W_i'$ such that both the diameters
of $W_i$, $W_i'$ tend to zero, and $W_i=P_k(W_i')$. Now, let $V_i$ be neighbourhoods of $g$ corresponding to $W_i'$. 
Since $g_n\to g$, for each $i$ there is some $g_{n_i}\in V_i$. We can choose $n_{i+1}>n_i$. 
Hence we can assume without loss of generality that $g_i\in V_i$. We can therefore find $\eta_i\in W'_i$ 
such that $\eta_i\inv C(g_i)\eta_i= C(g)$. Since $C(g_i)=C(h_i)$, $\eta_i\inv h_i\eta_i\in C(g)$ for all $i$,
and hence writing $\xi_i=\eta_i^k\in W_i$, $\xi_i\inv g_i\xi_i\in P_k(C(g))$. Since $\xi_i\to 1$ and $g_i\to g$,
$\xi_i\inv g_i\xi_i\to g$. Hence, by closedness of $P_k(C(g))$ (Proposition \ref{cartan pk closed}), 
we get $g\in P_k(C(g))$.\qed

{\centering\section{Characterization of the density of the power map $P_k$}}

In this section we prove Theorem \ref{dense criteria} by using Theorem \ref{pk closed}. 
Also, we deduce many characterization about the dense images of the power map, which can be used in the rest of the paper.

\noindent{\it Proof of Theorem~\ref{dense criteria}}: (a) $\Leftrightarrow$ (b) If $P_k(G)$ is dense in $G$ then 
$P_k(G)\cap \regg$ is dense in $\regg$, hence equal to $\regg$, by Theorem \ref{pk closed}. 
 As $\regg$ is dense in $G$ and it is contained in $P_k(G)$, the converse follows.

(b) $\Leftrightarrow$ (c) Suppose $\regg\subset P_k(G)$. It is enough to show that $C\cap \regg\subset P_k(C)$,
because $C\cap \regg$ is dense in $C$ (see \cite[Proposition 1.6]{H1}), while $P_k(C)$ is closed in $C$
(Proposition \ref{cartan pk closed}). 
If $x\in C\cap \regg$ then by assumption, there is $y\in G$ such that $x=y^k$. This implies $y\in\regg$, 
 by Lemma \ref{P_k regular}. Then $y$ lies in a unique Cartan subgroup, which has to be $C$ because $y^k\in C$. 
Conversely, any regular element $g$ can be found in a unique Cartan subgroup $C$. 
\qed

\begin{remark}\label{algebraic case}
For algebraic groups we can give a more straightforward argument.

(1) Let $G$ be an algebraic group over a field $\mathbb F$. Let $C$ be a Cartan subgroup of $G$. 
Then for $\mathbb F=\RR$ or $\QQ_p$, the regular elements of $G(\mathbb F)$ 
(group of $\mathbb{F}$-rational point of the algebraic group $G$) which can be conjugated to $C(\mathbb F)$ 
form an open subset of $G(\mathbb F)$ in the Euclidean or the ultrametric topology.

To see this, let $\Phi:G\times C\to G: (h,y)\mapsto hyh\inv$ and fixing $y\in C$, 
let $\Phi_y: G\to G$ is given by $\Phi_y(h)=hyh^{-1}.$ Note that the differential 
of $\Phi$ at $(e,y)$, where $y\in \regg\cap C$, is surjective. 
Write $T_yC=dL_y T_eC$, 
\[d\Phi (\xi, dL_y v)=d\Phi(\xi,0)+dL_yv=d\Phi_y(\xi)+dL_yv=dL_y(df(\xi)+v),\]
where $f:G\to G:h\mapsto y\inv h y h\inv$. Note that, $f$ is the composition 
of $G\to G\times G:h\mapsto (y\inv h y, h\inv)$ and $G\times G\stackrel\mu\to G:(x,y)\mapsto xy$. 
Then, 
\[df(\xi)=d\mu({\rm Ad}_y\inv \xi,-\xi)=({\rm Ad}_y\inv-I)\xi.\]
Now, since $y\in C$ is regular, $T_eC=N({\rm Ad}_y\inv -I)$. Therefore we get the following.
\[d\Phi(\frg\times T_yC)=dL_y({\rm im}({\rm Ad}_y\inv-I)+N({\rm Ad}_y\inv -I))=\frg.\]

(2) Now, for $\mathbb F=\QQ_p$, if the group is not anisotropic, there is a nonempty open subset of regular elements 
which can be conjugated to a Cartan subgroup $C$ which contains a copy of $\QQ_p^\times=\ZZ\times \ZZ_p^\times$. 
In this case, for any $k>0$, $P_k(C)$ is proper, open and closed in $C$. Therefore none of (a), (b), (c) of 
Theorem \ref{dense criteria} will hold. On the other hand, if $G$ is anisotropic then it is weakly exponential. 
In this case, density of $P_k$ is equivalent to surjectivity.

(3) For $\mathbb F=\RR$, if the group contains any disconnected Cartan subgroup $C$, then $P_{2k}(C)=C^*$, 
the identity component. Therefore, $P_{2k}$ is never dense. On the other hand, if $P_{2k+1}(C)=C$,
so that $P_{2k+1}(G)$ contains all regular elements, hence is dense. In case all Cartan subgroups are 
connected $G$ is weakly exponential. 
\end{remark}

\noindent{\it Proof of Corollary~\ref{density and w.e}}:
  Suppose that $P_k(G)$ is dense in $G.$ Then by Theorem \ref{dense criteria}, 
   $P_k(G)$ contains a dense open set (namely $\regg$).
  Hence by Baire category theorem, $\cap_{k\ge 1} P_k(G)$ contains a non empty open set and therefore $G$ is weakly exponential.
  Now the converse follows from McCrudden criterion.
 \hfill$\Box$
 
 Theorem \ref{dense criteria} can be used to obtain
the following criteria of weak exponentiality, which is well known (\cite[Theorem I.2]{N}).
\begin{corol}\label{Cartan subgroup}
Let $G$ be a connected Lie group and let $\frg$ be the Lie algebra associated to $G.$ The following are equivalent.
\begin{enumerate}
 \item[\rm (a)] $G$ is weakly exponential.
 \item[\rm (b)] $\regg\subseteq\exp(\frg).$
 \item[\rm (c)] All Cartan subgroups of $G$ are connected.
\end{enumerate}
\end{corol}
\begin{proof}
$(a)\Leftrightarrow (b)$ By Corollary \ref{dense criteria},
 $G$ is weakly exponential if and only if $P_k(G)$ is dense for all $k$, which is equivalent to saying that
 $\regg\subseteq\cap_{k\ge 1}P_k(G)\subseteq\exp(\frg)$, by Theorem \ref{dense criteria}.
 
 $(a)\Leftrightarrow (c)$ 
 By using Theorem \ref{dense criteria}, $G$ is weakly exponential
 if and only if all Cartan subgroups of $G$ are divisible (i.e., $P_k(C)=C$ for all $k\ge 1$).
 By \cite[Proposition 1]{HL}, this is equivalent to saying that $C$ is connected. 
\end{proof}

\begin{proposition}\label{density on levi part}
Let $G$ be a connected Lie group, and let $R$ be the radical of $G.$ Let $k>0$ be an integer.
Then $P_k:G\rightarrow G$ is dense if and only if $P_k:G/R\rightarrow G/R$ is dense.
\end{proposition}
 
\begin{proof}
The right implication is obvious.
For the converse, let $H$ be a Cartan subgroup of $G.$ Then there exists a Levi subgroup $S$
of $G$ such that $H=(H\cap S)(H\cap R).$
 We write, $G=S\cdot R$.
 Let $Z=S\cap R$ and $\Delta(Z):=\{(z,z^{-1})|z\in Z\}.$ Now, consider $S\ltimes R$ as an external semidirect product of $S$ and $R.$ 
 The map $\pi: S\ltimes R\rightarrow S\cdot R$ sending $(s,r)\to s\cdot r$, for $s\in S$, $r\in R$, is a covering map and
 $G$ can be thought of as a quotient of $S\ltimes R$ by $\Delta(Z).$  
  Then Cartan subgroup $\pi^{-1}(H)$ of $S\ltimes R$ is of the form
 $C\times N$ where $C=H\cap S$ is a Cartan subgroup of $S$, $N=H\cap R$ is connected nilpotent subgroup of $R$, and $H=(C\times N)/\Delta(Z)$.
  Also, note that as $S\rightarrow S/Z$ is covering, then $C/Z$ is a Cartan subgroup of $G/R.$
 Now, there is a short exact sequence of groups \[1\to N\to (C\times N)/\Delta(Z)\to C/Z\to 1.\]
 As $P_k(G/R)$ is dense in $G/R$, by Theorem \ref{dense criteria}, $P_k(C/Z)=C/Z.$ Now, let $x=\overline{x_1x_2}\in (C\times N)/\Delta(Z)$ for some
 $x_1\in C$ and $x_2\in N.$ As $P_k(C/Z)=C/Z$, there exists $y_1\in C$ and $z\in Z$ such that $x_1=y_1^kz.$
 Thus $\overline{x_1x_2}=\overline{y_1^kzx_2}=\overline{y_1^ky_2^k}$ in $H$, for some $y_2\in N$,
 as $z$ is contained in $N$ and $N$ is connected nilpotent Lie group. Now we notice that
 $\overline{x_1x_2}=\overline{y_1y_2}^k$, as $C$ and $N$ commute. 
 Hence the power map $P_k:H\to H$ is surjective. Therefore,
 by applying Theorem \ref{dense criteria}, the proposition follows.
 \end{proof}

 To deduce Corollary \ref{dense criteria 1}, we need next two lemmas.
  
 \begin{lemma}\label{abelian Cartan}
 Let $\pi:G\to H$ be a covering of connected Lie groups, with discrete kernel $D$. Let $C\subset H$ be a connected Cartan subgroup which is abelian. Then the corresponding Cartan subgroup $M=\pi\inv(C)$ is also abelian.
\end{lemma}
\begin{proof}
 Note that the identity component $M^*$ is a covering of $C$. Since $C$ is connected abelian, $M^*$ is abelian.
 Any element $x\in M$ is 
 of the form $x=mz$ for some $m\in M^*$ and $z\in D.$ Now we note that, the group homomorphism $M^*\times D\to M$,
 given by $(m',z')\to m'z'$, is surjective. Hence $M$ is abelian.
\end{proof}
 \begin{lemma}\label{C^* is central}
  Let $\tilde G$ be a connected Lie group with discrete center. Let $\tilde C$ be any Cartan subgroup of $\tilde G.$
  Then the connected component $\tilde C^*$ of $\tilde C$ is central. 
 \end{lemma}
\begin{proof}
 Let $G=\tilde G/Z(\tilde G)$ and $C:=\pi(\tilde C)$ be the corresponding Cartan subgroup of $G.$
 We can write $C=A\times F$, where $A$ is connected and
 $F$ is a $\mathbb{Z}/2\mathbb{Z}$ vector space. If $F$ is trivial, by Lemma \ref{abelian Cartan}, $\tilde C$ is abelian and hence
 the conclusion holds. Suppose that $F$ is not trivial. We observe that $\pi^{-1}(F)$ is a
 normal subgroup of $\tilde C$, where $\pi:\tilde C\rightarrow C$ is the natural projection map. Therefore, $\tilde C^*$
 acts trivially on $\pi^{-1}(F)$, as $\tilde C^*$ is connected and $\pi^{-1}(F)$ is discrete. Thus elements of 
 $\tilde C^*$ and $\pi^{-1}(F)$ commute each other. Note that $\tilde C$ is a quotient of $\tilde C^*$ and $\pi^{-1}(F).$
 Indeed, for $x\in\tilde C$, let $\pi(x)=(a,b)\in A\times F.$ We can choose $x_1\in\tilde C^*$ such that
 $x_1=\pi^{-1}(a).$ Then it follows that $xx_1^{-1}\in\pi^{-1}(F).$  Thus $x=(xx_1^{-1})x_1\in\pi^{-1}(F)\times\tilde C^*$
 and hence $\tilde C^*$ is central in $\tilde C.$
 \end{proof}
\noindent{\it Proof of Corollary~\ref{dense criteria 1}}:
 $(:\Rightarrow)$ This follows from Theorem \ref{dense criteria}.
 
 $(\Leftarrow:)$ Let $R$ be the solvable radical of $G.$ We recall that,
  there is a decomposition of the Cartan subgroup $C$ (\cite[Theorem 1.9 (ii)]{W}).
  For $C$, there is a Levi subgroup $S$ of $G$ such that
  $C=C_s\cdot N$ (almost direct
 product), where $C_s=C\cap S$ is a Cartan subgroup of $S$ and $N=C\cap R$ is connected nilpotent contained
 in $Z_G(S).$ Therefore we observe that $C/C^*\simeq C_s/C_s^*Z$, where $Z=C_s\cap N$, a central subgroup. Also, we note
 that if $C_1$ is a Cartan subgroup of $G/R$, then $C_1/C_1^*\simeq C_s/C_s^*Z.$ Hence $P_k:C/C^*\rightarrow C/C^*$
 is surjective if and only if $P_k:C_1/C_1^*\rightarrow C_1/C_1^*$ is surjective. Applying Lemma \ref{C^* is central}
 for the semisimple group $G/R$,
 we get $P_k:C_1\rightarrow C_1$ is surjective and hence by Theorem \ref{dense criteria}, $P_k:G/R\rightarrow G/R$ is dense.
 Now by Proposition \ref{density on levi part}, we conclude that $P_k(G)$ is dense in $G.$
\hfill$\Box$

\section{Density of the power maps on simple Lie groups}

In this section we study the dense images of the power maps in simple Lie groups. We use results in \S 3, 
to determine the integer $k$, for which $P_k$ has dense images.

 We recall from Remark \ref{algebraic case} that, for a connected real algebraic group $G$ with disconnected Cartan subgroup, 
 $P_k$ is dense if and only if $k$ is odd.
 However, in general this is not true.
 
 \begin{example}
  Let $\tilde G=\widetilde{\rm SL(2,\mathbb{R})}$ be the simply connected cover of $\rm SL(2,\mathbb{R}).$ Then
  $P_k:\widetilde{\rm SL(2,\mathbb{R})}\rightarrow\widetilde{\rm SL(2,\mathbb{R})}$ is not dense for all $k>1.$ 
  Indeed, the center of $\tilde G$ is infinite cyclic, and the adjoint group $\rm PSL(2,\RR)$ has a Cartan 
  subgroup isomorphic to $\RR$. The corresponding Cartan subgroup in $\tilde G$ is therefore isomorphic
  to $\RR\times\ZZ$, where the component $\ZZ$ is contributed by the center. Therefore for $k>1$, $P_k$ is never 
  surjective on this Cartan subgroup, so the assertion follows by Theorem \ref{dense criteria}.
 \end{example}

There is a list of weakly exponential non-compact non-complex simple Lie algebras in Neeb (\cite{N}). 
Neeb also mentioned exactly which in the list are completely weakly exponential (i.e., whose corresponding simply connected group is weakly exponential).
  Thus recalling Borel's criterion about weak exponentiality,
  one can understand the connectedness of the Cartan subgroup.

 For a connected Lie group $G$, if Cartan subgroups are connected, then $P_k$ is dense in $G$ for all $k$, as $G$
 is weakly exponential.
 Therefore we concentrate the case when Cartan subgroups are disconnected. As we noted before 
 (Proposition \ref{density on levi part}), the study of density boils down 
 to semisimple Lie group, at this moment
 we now investigate this phenomenon for simple Lie groups. 
 
 Let us fix some notations for this section.
 Let $\tilde G$ be the universal cover of the connected simple adjoint Lie group $G.$ 
 Let $\pi$ denote the covering map and $\tilde C$ be a Cartan
 subgroup of $\tilde G.$ Then let $C=\pi(\tilde C)$ be the corresponding Cartan subgroup of $G.$
 Let us denote the fundamental group of $G$ by $\pi_1(G).$ So
 we get a natural short exact sequence
  \[1\to\pi_1(G)\to \tilde C\to C\to 1.\]
  
  For a connected simple adjoint Lie group $G$, we notice that its fundamental group is finitely generated,
  whose possible factors are isomorphic to $\mathbb{Z}$,
  $\mathbb{Z}/2\mathbb{Z}$, $\ZZ/4\ZZ$, $\ZZ/8\ZZ$, or 
  $\mathbb{Z}/3\mathbb{Z}$. We discuss some cases. In the following cases we assume that some Cartan subgroup of $\tilde G$,
  say $\tilde C$ is disconnected.

\noindent {\bf Case 1.} 
 Let $\pi_1(G)\simeq\mathbb{Z}_2$ or non cyclic group of order $4.$

This happens for $A_n I$ ($n\ge 2$), $B_n I$, $E^6_6 I$, $E_8^8 VIII$, $F_4^4 I$, $G_2^2 I$, $C_n II$, $E_7^{-5} VI$, $E_8^{-24} IX$, $F_4^{-22} II$. Note that none of these is weakly exponential.

\begin{lemma}\label{P1}
  
  In Case 1, $P_k:\tilde G\rightarrow\tilde G$ is dense if and only if $k$ is odd.
 \end{lemma}
 \begin{proof}
 First assume that the Cartan subgroup $C=\pi(\tilde C)$ of $G$ is disconnected.
  By hypothesis, it is enough to check for odd $k$, as for even $k$, $P_k(G)$ is not dense in $G.$ Note that
  $\pi_1(G)$ is central in $\tilde C.$ Since for $k$ odd, $P_k$ is surjective on $C$ and $\pi_1(G)$ as well,
  $P_k:\tilde C\to\tilde C$ is surjective and hence by Theorem \ref{dense criteria}, the
  lemma follows.
  
  Now, suppose that $C$ is connected. Using the same argument as before, for odd $k$, $P_k$ is dense in $\tilde G.$ Now Suppose that
  $P_2$ is dense in $\tilde G.$ Then $P_k$ is dense in $\tilde G$ for all $k$, which further implies that $P_k(\tilde C)=\tilde C$
  for all $k$, by Theorem \ref{dense criteria}. Thus by \cite[Proposition 1]{HL}, $\tilde C$ is connected, which is a contradiction.
  This implies that, for even $k$, $P_k(\tilde G)$ is not dense in $\tilde G$, as desired.
 \end{proof}

\noindent {\bf Case 2(a).}
 $C$ is connected, and $\pi_1(G)\simeq\mathbb{Z}.$
This occurs for $E_7^{-25}VII$, $A_1 I$.

As $C$ is connected, we get the following short exact sequence
\[1\to \tilde C^*\cap\mathbb{Z}\to \tilde C^*\to C\to 1.\] Suppose that $\tilde C^*\cap\mathbb{Z}$ is the subgroup
$n\mathbb{Z}$ of $\mathbb{Z}.$ Thus we have $\tilde C/\tilde C^*\simeq\mathbb{Z}/n\mathbb{Z}.$ 

 In this case, $P_k:\tilde G\rightarrow\tilde G$ is dense if and only if $k$ is coprime to $n.$

 Indeed, as $P_k:\mathbb{Z}/n\mathbb{Z}\rightarrow\mathbb{Z}/n\mathbb{Z}$ is surjective if and only if $k$ is coprime to $n$,
 the result follows from Theorem \ref{dense criteria 1}.

\vskip 1em
\noindent {\bf Case 2(b).}
$C$ is disconnected, and  $\pi_1(G)\simeq\mathbb{Z}.$
Then there is a surjective group homomorphism $\tilde C/\tilde C^*\to C/C^*$. Let $D={\rm Ker}(\tilde C/\tilde C^*\to C/C^*).$ 
As $\pi_1(G)$ is infinite cyclic, $D$ is isomorphic to
$\mathbb{Z}/n\mathbb{Z}$ for some $n.$

This happens for $A_1 I$, $C_n I$ ($n\ge 3$), $D_nIII$ ($n\ge 4$ even).

\begin{lemma}\label{P4}
  In case 2(b), $P_k:\tilde G\rightarrow\tilde G$ is dense if and only if $k$ is 
  odd and coprime to $n.$
 \end{lemma}
 \begin{proof}
 First observe that, as $C$ is disconnected, for even $k$, $P_k$ is not dense.
 
  In this case we get the following short exact sequence.
 \[1\to \mathbb{Z}/n\mathbb{Z}\to \tilde C/\tilde C^*\to C/C^*\to 1.\]
 Let $\sigma$ be any element in $\mathbb{Z}/n\mathbb{Z}.$ If $\pi:\tilde C/\tilde C^*\to C/C^*$, then $\pi(\sigma)=0.$
 Suppose that $\sigma=t^k$ for some $t\in\tilde C/\tilde C^*.$ Then we have 
 $0=\pi(\sigma)=\pi(t^k)=\pi(t)^k.$ As for odd integer, $P_k$ is identity on $C/C^*$, we conclude that 
 $t\in\mathbb{Z}/n\mathbb{Z}.$ Therefore, if $k$ is not coprime to $n$, $P_k$ is not surjective on
 $\tilde C/\tilde C^*.$ On the other hand, if $k$ is odd and coprime to $n$, then both 
 $P_k:\mathbb{Z}/n\mathbb{Z}\rightarrow\mathbb{Z}/n\mathbb{Z}$ and $P_k:C/C^*\rightarrow C/C^*$ are surjective and
 the result follows from the short exact sequence.
 \end{proof}
 
 \noindent {\bf Case 3.} $C$ is not connected and $\pi_1(G)\simeq\mathbb{Z}/6\mathbb{Z}$. This happens for $E_6^2II.$
  \begin{lemma}
  $P_k:\tilde G\rightarrow\tilde G$ is dense if 
   $k$ is odd and coprime to $3$ and for even $k$, $P_k:\tilde G\to\tilde G$ is not dense.
   
   \end{lemma}
   \begin{proof}
   We first note that, if $k$ is odd and coprime to $3$ (hence to $6$), then $P_k$ is surjective both on 
 $\mathbb{Z}/6\mathbb{Z}$ and $C.$ Hence $P_k$ is surjective on $\tilde C.$ 
 Then the result follows by applying Theorem \ref{dense criteria}.
 
 Since $C$ is disconnected, $P_2$ can not be dense in $\tilde G.$
 \end{proof}

 \noindent {\bf Case 4.}  
   Let $\pi_1(G)\simeq\mathbb{Z}/4\mathbb{Z}$ or $\mathbb{Z}/8\mathbb{Z}$ and $C$ is disconnected. 
   This occurs for $D_nI$($n\ge 4$ and $n$ odd) and $D_nI$($n\ge 4$ and $n$ even) respectively.
   The following lemma is very straightforward and hence we omit the details.
   \begin{lemma}
   In case 4, $P_k:\tilde G\to\tilde G$ is dense 
   if and only if $k$ is odd. Moreover, if $G'$ be any cover of $G$, then $P_k:G'\to G'$ is dense if and only if $k$ is odd. 
   \end{lemma}
   
 \noindent {\bf Case 5.} 
 $G$ is split. 

As usual, let $\tilde G$ be its universal cover. Let $Z$ denote the discrete group $\pi_1(G).$
Suppose that there is a Cartan subgroup $C$ of the form $A\times F$, 
where $A$ is simply connected and $F$ is a $\mathbb{Z}/2\mathbb{Z}$ vector space.
Then we get the following short exact sequence,
\[1\to Z\to \tilde C\to A\times F\to 1,\] where $\tilde C$ is the corresponding Cartan subgroup of $C.$ 
 Let $K$ be the kernel of the map $\tilde C\to A.$ Then we have the following two short exact sequences
 \[1\to Z\to K\to F\to 1.\] and
 \[1\to K\to \tilde C\to A\to 1.\]
 Now, as $A$ is simply connected, we get a section of the above short exact sequence, so that $\tilde C=A\ltimes K.$
 It turns out that $\tilde C$ is a direct product of $A$ and $K$, as $A$ is connected
 and $K$ is discrete. Hence $P_k(\tilde C)=\tilde C$ if and only if $P_k(K)=K.$
 Furthermore, let $\sigma\in Z.$ Let $\sigma=t^k$ for some $t\in K.$ We denote $\pi$ is the map
 from $K\to F$ in the above short exact sequence. Note that for even $k$, $P_k(\tilde G)$ is not dense. For odd $k$, 
 as $P_k:C\to C$ is identity,
 \[0=\pi(\sigma)=\pi(t^k)=P_k(\pi(t))=\pi(t),\]
 we get $t\in Z.$ Thus we have $P_k(\tilde G)$ is dense in $\tilde G$ if and only if $P_k(Z)=Z.$
 
 Let $D$ be a subgroup of $Z$
 and $G'=\tilde G/D.$ Then by above procedure we can conclude that
 $P_k(G')$ is dense in $G'$ if and only if $P_k(Z/D)=Z/D.$ 
 
 \begin{remark}
   In Case 3, the authors have not addressed the case when $k$ is odd and divisible by $3$, 
  as they are not able to calculate the number of connected components of the
 Cartan subgroups involved. They have also left unaddressed the case when $G$ is a non-split, non-compact simple Lie group and
 $\pi_1(G)$ contains an infinite cyclic subgroup, for the same reason.
\end{remark}

 \section{Relation between square map and weak exponentiality}

 In this section we discuss the relation between the density of $P_2$ and weak exponentiality of Lie groups. Also, we
 prove Corollary \ref{square map}.\\

\noindent{\it Proof of Corollary~\ref{square map}}:
Let $R$ be the solvable radical of $G.$ Then $G/R$ is a connected linear semisimple Lie group. 
  As any connected semisimple Linear Lie group is precisely the Hausdorff connected 
  component of some Zariski connected semisimple real algebraic
 group, Cartan subgroups of a connected linear Lie group are either connected or it has $2^m$ (for some positive integer $m$)
 number of disconnected components. Let $C$ be any Cartan subgroup and $C^*$ be its connected component.
 As $P_2(G)$ is dense in $G$, by Corollary \ref{dense criteria 1}, we get $P_2:C/C^*\to C/C^*$ is surjective.
 This implies that $C=C^*$, and hence $G$ is weakly exponential.
 
 The converse follows from Corollary \ref{density and w.e}.
  \qed

 \begin{example}\label{E}
We give an example of a non-linearizable Lie group $G$ which is not weakly exponential, but $P_2(G)$ is dense.
  Let us consider $G=\rm PSL(2,\RR)$ and $\tilde G=\widetilde{\rm SL(2,\RR)}.$
  We know that $Z(\tilde G)$ is infinite cyclic. Fix a generator $\sigma$. 
  Let $Z$ to be the subgroup of $Z(\tilde G)$ generated by $3\sigma$ and define $G'=\tilde G/Z.$ 
  There is a Cartan subgroup of $G$ isomorphic to $\RR$, and the corresponding Cartan subgroup for $\tilde G$ is 
  isomorphic to $\RR\times \ZZ$, where the factor $\ZZ$ is contributed by the center. Therefore, the corresponding 
  Cartan subgroup of $G'$ would be isomorphic to $\RR\times \ZZ/3\ZZ$, which is not connected.
  By Theorem \ref{dense criteria}, $P_2(G')$ is dense in $G'$, but it has a disconnected Cartan subgroup, and hence
  $G'$ is not weakly exponential.
\end{example}  

\begin{remark}
Let $G$ be a connected Lie group, not necessarily linear, and let $R$ be its radical. 
Suppose the adjoint semisimple group ${\rm Ad}(G/R)$ does not contain any simple factor whose fundamental 
group contains an infinite cyclic subgroup. It follows from the cases discussed in the previous section (\S 4) that 
if $P_2(G)$ is dense then $G$ is weakly exponential. 
\end{remark}
  
 \section{Density on full rank subgroup}
  We recall that a Lie subgroup $A$ of $G$ is of \emph{full rank} if its Lie algebra $\mathfrak a$ contains
 a Cartan subalgebra of $\mathfrak g={\rm Lie}(G).$ In this case, every Cartan subalgebra of $\fra$ is a 
 Cartan subalgebra of $\frg$ (\cite[Proposition 3, \S 3, Ch. VII]{B}).
 
\begin{proposition}\label{maximal rank odd integer}
Let $G$ be a connected Lie group and let $A$ be a connected subgroup of full rank. 
Let $k\ge 1$ be an odd integer. If $P_k(G)$ is dense in $G$, then $P_k(A)$ is dense in $A$.
\end{proposition}
\begin{proof} 
 Note that since $A$ is of full rank, $Z(G)^*\subset A$. By killing $Z(G)^*$, it is easy to see that $A_1=A/Z(G)^*$
is a full rank subgroup of $G_1=G/Z(G)^*.$ Inductively, set $Z_i=Z(G_i)^*$ and $A_i=A_{i-1}/Z_{i-1}.$ 
Then $A_i$ is of full rank in $G_i$. Since the dimension of $G$ is finite, $Z_i$ is discrete for some $i$. 
It is enough to show that $P_k(A_i)$ is dense in $A_i$. For, if $H$ is a connected Lie group and $Z\subset H$ is 
central connected, then $P_k(H)$ is dense in $H$ if and only if $P_k(H/Z)$ is dense in $H/Z$. This is because, 
the projection $\pi:H\to H/Z$ is a submersion, hence an open map, and $\pi\inv(P_k(H/Z))=P_k(H)$. 

Thus we can assume that $Z(G)$ is  
discrete. Since $Z(G)$ is the fundamental group of
linear Lie group ${\rm Ad}(G)$, it is a finitely generated abelian group. 

We know that $\regg\cap A$ is dense in $A$ 
(see proof of \cite[I.6]{N}). It is enough to prove that elements in $\regg\cap A$ admit $k$-th roots in $A$. 
Let $g\in \regg\cap A$. There is a unique Cartan subgroup $C$ of $G$ such that $g\in C$. The subgroup $C':=C\cap A$ 
contains the identity component $C^*$, which is also the identity component of the
Cartan subgroup of $A$ that contains $g$, with Lie algebra given by $N_\fra(Ad_g-I)=N_\frg(Ad_g-I)$. 
Therefore, $C'$ is the disjoint union of cosets of $C^*$, hence open in $C'$. It is enough to show that $P_k(C')=C'$.

 Writing $D=Z(G)$, let $\pi:G\to G/D$ be the projection. Now, by Theorem \ref{dense criteria}, $P_k(C)=C$. We get an exact
sequence as follows, where $E$ is the corresponding Cartan subgroup of $G/D$: \[1\to D\to C\stackrel\pi\to E\to 1.\]
This gives another short exact sequence \[1\to D/(D\cap C^*)\to C/C^*\stackrel\pi\to E/E^*\to 1.\]However, $E/E^*$ is 
an $\ZZ/2\ZZ$-vector space, so for odd $k\ge 0$, $P_k$ induces identity on $E/E^*$. Therefore $P_k$ is surjective 
on $D/(D\cap C^*)$. Indeed, let $x\in D/(D\cap C^*)$. Since $P_k(C)=C$, there is some $y\in C/C^*$ such that $x=y^k$. 
But $0=\pi(x)=\pi(y)^k=\pi(y)$, so $y\in D/(D\cap C^*)$. This means, $D/(D\cap C^*)$ is a torsion abelian group with 
torsion coprime to $k$. However, $C'$ fits into a short exact sequence \[1\to D'\to C'\to F\to 1,\]where $F=\pi(C')$ 
is an open subgroup of $E$, and $D'=D\cap C'$. This gives rise to the short exact sequence as follows. 
\[1\to D'/(D'\cap C^*)\to C'/C^*\to F/F^*\to 1.\]
Here $F/F^*$ is the group of connected component of $F\subset E$, 
which must be an $\mathbb F_2$-vector space, and $D'/(D'\cap C^*)$ is a subgroup of $D/D\cap C^*$, hence a torsion 
abelian group with torsion coprime to $k$. Therefore $P_k$ is surjective on $C'/C^*$.
\end{proof}
 
\begin{lemma}\label{maximal rank on abelian Cartan}
 Let $G$ be a connected Lie group and let $A$ be a connected full rank subgroup. Let $C$ be any Cartan subgroup 
 and $C^*$ be its connected component. Suppose that
 $C/C^*$ is finitely generated abelian.
 Let $k>0$ be an integer. If $P_k:G\rightarrow G$ is dense then $P_k:A\rightarrow A$ is dense.
\end{lemma}
\begin{proof}
In view of Proposition \ref{maximal rank odd integer}, it is enough to prove for integer $k=2.$ 
As $P_2(G)$ is dense in $G$, by Corollary \ref{dense criteria 1}, $P_2:C/C^*\rightarrow C/C^*$ is surjective.
By hypothesis $C/C^*$ is a torsion  abelian group and torsion is coprime to $2.$ We recall from the
proof of Proposition \ref{maximal rank odd integer} that, to prove $P_k:A\rightarrow A$ is dense, it only
needs to show that $P_2:A\cap C/C^*\rightarrow A\cap C/C^*$ is surjective. Note that $A\cap C/C^*$ is also torsion subgroup of $C/C^*$
and torsion is coprime to $2$ as well. This proves the lemma.
\end{proof}
\begin{proposition}\label{maximal rank even integer}
  Let $G$ be a connected Lie group and let $A$ be a connected full rank subgroup of $G.$ Then $P_2:G\to G$ is dense
  implies $P_2:A\to A$ is dense.
 \end{proposition}
\begin{proof}
Arguing as before (see the proof of Proposition \ref{maximal rank odd integer}), we may assume that $Z(G)$ is discrete. Now,
 let $\pi: G\to {\rm Ad}(G)$ be the covering map and $C$ be a Cartan subgroup of $G.$ Let $C_{\rm Ad}=\pi(C)$ be the corresponding
 Cartan subgroup of ${\rm Ad}(G).$ Since $P_2(G)$ is dense in $G$, $P_2:{\rm Ad}(G)\to {\rm Ad}(G)$
 is dense and hence $C_{\rm Ad}$ is connected.
 Let $C^*$ denote the identity component of $C.$ We observe that the restriction map $\pi:C^*\to C_{\rm Ad}$ is surjective.
 Indeed, $\pi(C^*)$ is non empty open and closed in $C_{\rm Ad}$ (for details, see the proof of Lemma \ref{B}).
 Therefore we get a surjective homomorphism from
 $Z(G)\times C^*\to C_{\rm Ad}.$ Hence $C/C^*$ is finitely generated abelian group (as $Z(G)$ is finitely generated abelian group).
 Thus by Lemma \ref{maximal rank on abelian Cartan}, it follows that $P_2:A\to A$ is dense.
 \end{proof}
 
\noindent{\it Proof of Theorem~\ref{full rank}}:
The theorem follows from Proposition \ref{maximal rank odd integer} and Proposition \ref{maximal rank even integer}.
\qed\\

\noindent{\it Proof of Corollary~\ref{dense criteria 2}}:
 It is noted in Neeb (\cite{N}) that, $Z_G(t_a)$ is a connected, full rank, reductive subgroup of $G.$ Therefore forward implication
 of the statement follows from Theorem \ref{full rank}. 
 Now for the converse, let $C'$ be any Cartan subgroup of $G.$ Let $\mathfrak{c'}$ be the Lie algebra associated to $C'.$
 Since $\mathfrak{g}$ is reductive, $\mathfrak{c'}$ is reductive in $\mathfrak{g}$ and its associated Cartan subgroup
 is $C'=Z_G(\mathfrak{c'}).$ We may assume that
 $\mathfrak c'= t'_a +\mathfrak {a'}$,
 where $\mathfrak a' \subseteq \mathfrak a$ and $t'_a\supseteq t_a.$
 Then it is immediate
 that $C'\subseteq Z_G(t_a)$ and by \cite[Lemma I.5]{N}, $C'$ is a Cartan subgroup of $Z_G(t_a).$
 Now, by applying Theorem \ref{dense criteria} on $Z_G(t_a)$, we conclude that $P_k(C')=C'$, as required.
 \qed\\

{\bf Acknowledgement.} We thank Professor S. G. Dani for his useful comments, particularly Remark \ref{algebraic case}.

\vskip 2em
\parindent=0pt

{\small Saurav Bhaumik 

Department of Mathematics, 

Indian Institute of Technology Bombay, 

Powai, Mumbai 400076, India

\texttt{saurav@math.iitb.ac.in}
\vskip 1em

Arunava Mandal

Department of Mathematics, 

Indian Institute of Technology Bombay, 

Powai, Mumbai 400076, India

\texttt{amandal@math.iitb.ac.in}

}

\begin{thebibliography}{100}
 
\bibitem[B]{B} Bourbaki, N. \emph{Lie Groups and Lie Algebras}, Elements of Mathematics, Springer. Chapters 7-9. 
  \bibitem[C1]{C1} Chatterjee, P. \emph{On surjectivity of the power maps of solvable Lie groups}, J. Algebra 248 (2002), no. 2, 669 - 687.
  
 
  \bibitem[C2]{C2} Chatterjee, P. \emph{On the power maps, orders and exponentiality of p-adic algebraic groups}, J. Reine Angew. Math., (2009), no. 629, 201 - 220.
  
  \bibitem[C3]{C3} Chatterjee, P. \emph{Automorphism invariant Cartan subgroups and power maps of disconnected groups}, Math. Z. 269 (2011), no. 4, 221 - 233.

  \bibitem[DM]{D-M} Dani, S.G. and Mandal, A. \emph{On the surjectivity of the power maps of a class of solvable groups}, Preprint.
  \bibitem[F]{F} Fell, J. \emph{A Hausdorff topology for the closed subsets of a locally compact non Hausdorff space},

  \bibitem[H]{H} Hofmann, K.H. \emph{A memo on the exponential function and regular points}, Arch. Math., (1992), no. 59, 24 - 37.
  \bibitem[H1]{H1} Hofmann, K.H. \emph{Near-Cartan algebras and groups}, Seminar Sophus Lie, (1992), no. 2, 135 - 151.
  
\bibitem[HL]{HL} K.H. Hofmann, J.D. Lawson, \emph{Divisible subsemigroups of Lie groups},
J. London Math. Soc. 27 (1983), 427 - 437.
\bibitem[HM]{H-M} Hofmann, K.H., and A. Mukherjea, \emph{On the density of the image of the exponential function}, 
  Math. Ann., (1978), 263 -273.

\bibitem[Mc]{Mc} McCrudden, M. \emph{On $n$th roots and infinitely divisible element}, Arch. Math., (1992), no. 59, 24 - 37.
  

\bibitem[N]{N} Neeb, K-H. \emph{Weakly exponential Lie groups}, J. Algebra. 179 (1996), no. 15, 331-361.

     \bibitem[St]{St} Steinberg, R. \emph{On the power maps in algebraic groups}, Math. Res. Letters, (2003), no. 10, 621 - 624.
\bibitem[W]{W} W\"ustner, M. \emph{A generalization of the Jordan decomposition}. Forum Math. volume 15 (2003) 395-408.
  \end{thebibliography}
\end{document}